\documentclass[11pt]{article}

\usepackage[utf8]{inputenc}
\usepackage[T1]{fontenc}
\usepackage{lmodern}
\usepackage[margin=1in]{geometry}
\usepackage{amsmath,amssymb,amsthm}
\usepackage{mathtools}
\usepackage{microtype}
\usepackage[numbers]{natbib}
\usepackage{hyperref}

\hypersetup{
  colorlinks=true,
  linkcolor=blue,
  citecolor=blue,
  urlcolor=blue
}

\newtheorem{theorem}{Theorem}[section]
\newtheorem{proposition}[theorem]{Proposition}
\newtheorem{lemma}[theorem]{Lemma}
\newtheorem{corollary}[theorem]{Corollary}
\newtheorem*{theoremA}{Theorem A}
\newtheorem*{theoremB}{Theorem B}
\newtheorem*{theoremC}{Theorem C}
\theoremstyle{definition}
\newtheorem{definition}[theorem]{Definition}

\theoremstyle{remark}
\newtheorem{remark}[theorem]{Remark}

\newcommand{\Ball}{\mathbb B}
\newcommand{\C}{\mathbb C}
\newcommand{\cR}{\mathcal R}
\newcommand{\D}{\mathbb D}
\newcommand{\M}{\mathbb M}
\newcommand{\N}{\mathbb N}
\newcommand{\dd}{\mathrm{d}}
\newcommand{\Aut}{\operatorname{Aut}}

\newcommand{\Comm}{\operatorname{Comm}}

\title{Von Neumann Type Inequalities for NC Operator Balls}
\author{Ran Kiri}
\date{}

\begin{document}
\maketitle

\begin{abstract}
We establish finite-dimensional von Neumann type inequalities for commuting
matrix tuples in noncommutative operator balls. Let $\Omega=\D_Q$ be the NC
operator ball of a linear pencil $Q$ with linearly independent coefficients.
If $\Omega(1)$ admits a transitive subgroup of automorphisms with NC lifts and
$\Omega(1)$ has the bounded $H^\infty$-Gleason property at the origin, then
$\|f(T)\|\leq C_{\Omega,n}\|f\|_{H^\infty(\Omega(1))}$ for every commuting
tuple $T\in\Omega(n)$, with
$C_{\Omega,n}=(2\rho_{\Omega}K_{\Omega})^{n-1}$; here $K_{\Omega}$ is the
Gleason constant and $\rho_{\Omega}$ is the sum of the suprema of the
coordinate functions on $\Omega(1)$. The scalar coordinate bounds persist at
every matrix level. The hypotheses hold for products of row
and column balls, where $\rho_{\Omega}=d$, yielding an unconditional estimate
$(2dK_{\Omega})^{n-1}$. For a general type~I Cartan domain $R_{p,q}$, we
obtain only a conditional result: the same bound holds provided the scalar
domain has the bounded $H^\infty$-Gleason property at the origin, which we
do not establish when $\min(p,q)\geq 2$.
\end{abstract}

\section*{Introduction}

The classical von Neumann inequality states that if $T$ is a contraction on a
Hilbert space and $p$ is a polynomial, then
$\|p(T)\|\leq\sup_{|z|<1}|p(z)|$. For several commuting contractions the
corresponding inequality is more delicate: it holds for two commuting
contractions by Ando's theorem \cite{ando1963pair}, but fails in general for
three or more commuting contractions, by the example of Kaijser and Varopoulos
\cite{varopoulos1974inequality}; relatedly, Parrott showed that three commuting
contractions need not admit a commuting unitary dilation
\cite{parrott1970unitary}. If the matrix
size is fixed, however, one may still hope for a von Neumann inequality up to a
constant. In the Euclidean row ball this finite-dimensional viewpoint was
developed by Hartz--Richter--Shalit, who proved that for every fixed matrix
size $n$ there is a constant controlling all commuting row-contractive
$n\times n$ tuples, uniformly in the number of variables
\cite{hartzrichtershalit2021row}. This suggests asking whether analogous
finite-dimensional estimates can be established for other natural graded
matrix domains such as noncommutative (NC) operator balls.

We give sufficient conditions for finite-dimensional von Neumann type
inequalities on noncommutative operator balls. Throughout, $Q$ is a
\emph{linear pencil}
\[
  Q(z)=A_1z_1+\cdots+A_dz_d,
  \qquad A_1,\ldots,A_d\in\M_{r,c}(\C)
\]
with linearly independent coefficients, and $\Omega=\D_Q$ is the associated
noncommutative operator ball. The examples treated below --- row and column
balls, Cartan domains of type~I, and their products --- arise from linear
pencils. NC operator balls in this setting have been studied in
\cite{sampatshalit2025classification,sampatshalit2025weakstar}, and von Neumann
type inequalities for such domains in
\cite{kalyuzhniy1997linear,ambrozie2003von}. The levels $\Omega(n)$ identify
with the open unit balls of $\M_n(E_Q)$, where $E_Q$ is the operator space
spanned by the coefficients. Moreover, the scalar coordinate bounds hold at
every matrix level; see Lemma~\ref{lem:matrix-coefficients}.

For commuting tuples $T\in\Omega(n)$, we seek estimates of the form
\[
  \|f(T)\|\leq C_{\Omega,n}\|f\|_{H^\infty(\Omega(1))}
\]
for every $f\in H^\infty(\Omega(1))$, where $C_{\Omega,n}$ depends only on the
NC ball and on $n$. Here $H^\infty(\Omega(1))$ is the classical Banach algebra
of bounded holomorphic functions on the scalar domain, and $f(T)$ is defined by
the Taylor functional calculus on commuting tuples.

Our main theorem gives sufficient conditions for such a von Neumann type
inequality. We require:

\medskip
\noindent\textbf{(A1) Transitive automorphisms.}
There is a subgroup $G\subseteq\Aut(\Omega(1))$ acting transitively on
$\Omega(1)$, such that every $\theta\in G$ extends to an NC automorphism of
$\Omega$.

\medskip
\noindent\textbf{(A2) Gleason at the origin.}
The scalar domain $\Omega(1)$ has the bounded $H^\infty$-Gleason property at
the origin: there is a constant $K_{\Omega}>0$ such that every
$f\in H^\infty(\Omega(1))$ with $f(0)=0$ admits a decomposition
\[
  f(z)=\sum_{i=1}^d z_i g_i(z),
  \qquad
  \|g_i\|_\infty\leq K_{\Omega}\|f\|_\infty,
\]
where $g_i\in H^\infty(\Omega(1))$ for $i=1,\ldots,d$. Set
\[
  \rho_{\Omega}:=\sum_{i=1}^d\ \sup_{z\in\Omega(1)}|z_i|.
\]

\begin{theoremA}
Let $\Omega=\D_Q$ be an NC operator ball satisfying \textup{(A1)} and
\textup{(A2)} above.
Then, for every commuting tuple $T\in\Omega(n)$ and every
$f\in H^\infty(\Omega(1))$,
\[
  \|f(T)\|\leq C_{\Omega,n}\|f\|_{H^\infty(\Omega(1))}
\]
with $C_{\Omega,1}=1$ and
$C_{\Omega,n}=(2\rho_{\Omega}K_{\Omega})^{n-1}$ for $n\geq 2$.
\end{theoremA}

In all the examples below every coordinate function has supremum $1$ on the
scalar level, so that $\rho_{\Omega}=d$ and the constant takes the form
$(2dK_{\Omega})^{n-1}$.

We then apply the abstract criterion to NC operator balls associated with
Cartan domains of type~I. Write
\[
  R_{p,q}:=\{Z\in\M_{p,q}(\C):\|Z\|<1\}
\]
for $p,q\in\N$. The domains $R_{1,q}$ and $R_{p,1}$ are, respectively, row
and column balls. We show that products of such factors form NC operator balls
satisfying the assumptions of Theorem~A.

\begin{theoremB}
Let
\[
  \Omega_1
  =
  R_{p_1,q_1}\times\cdots\times R_{p_s,q_s}
  \subset\C^d,
\]
where $d=\sum_{j=1}^s p_jq_j$, and assume that for each $j=1,\ldots,s$ either
$p_j=1$ or $q_j=1$. Let $\cR_{\Omega}$ be the corresponding product NC operator
ball. Then for every commuting $T\in\cR_{\Omega}(n)$ and every
$f\in H^\infty(\Omega_1)$,
\[
  \|f(T)\|\leq C_{\Omega,n}\|f\|_{H^\infty(\Omega_1)}
\]
with $C_{\Omega,n}=(2dK_{\Omega})^{n-1}$, where $K_{\Omega}$ depends only on
$\Omega_1$.
\end{theoremB}

For every type~I Cartan domain $R_{p,q}$, the associated NC operator ball
$\cR_{p,q}$ satisfies \textup{(A1)}. The bounded $H^\infty$-Gleason property
\textup{(A2)} holds when $\min(p,q)=1$. For $\min(p,q)\geq 2$, we are not
aware of a proof of this property in the literature. We therefore state the
following result under hypothesis \textup{(A2)}.

\begin{theoremC}
Let $p,q\geq 1$, and suppose $R_{p,q}$ satisfies assumption
\textup{(A2)}. Then, for every commuting $T\in\cR_{p,q}(n)$ and every
$f\in H^\infty(R_{p,q})$,
\[
  \|f(T)\|\leq C_{p,q,n}\|f\|_{H^\infty(R_{p,q})}
\]
with $C_{p,q,n}=(2pqK_{p,q})^{n-1}$, where $K_{p,q}$ depends on $R_{p,q}$.
\end{theoremC}

This paper is organized as follows. Section~\ref{sec:prelim} recalls the
bounded $H^\infty$-Gleason property, standard NC-set terminology, NC operator
balls, and commuting loci.
Section~\ref{sec:recipe} states and proves the abstract sufficient criterion,
Theorem~A, and in Subsection~\ref{subsec:fd-representations} uses it to
describe, up to similarity, all bounded finite-dimensional representations of
the scalar uniform algebra on~$\Omega(1)$.
Section~\ref{sec:gleason-product} contains the product Gleason
proposition. Section~\ref{sec:row-column} applies the criterion to products of
row and column balls and proves Theorem~B. Section~\ref{sec:cartan} records the
conditional consequence for general Cartan domains.

\section{Preliminaries}\label{sec:prelim}

\subsection{The bounded \texorpdfstring{$H^\infty$}{H-infinity}-Gleason property}

\begin{definition}\label{def:Hinfty-Gleason}
Let $\Omega\subset\C^m$ be open and $z_0\in\Omega$. We say that $\Omega$ has
the \emph{bounded $H^\infty$-Gleason property at $z_0$} if there exists
$K_{\Omega}>0$ such that whenever $f\in H^\infty(\Omega)$ vanishes at $z_0$, there are
$g_1,\ldots,g_m\in H^\infty(\Omega)$ with
\[
  f(z)=\sum_{i=1}^m \bigl(z_i-(z_0)_i\bigr)g_i(z),
  \qquad
  \|g_i\|_\infty\leq K_{\Omega}\,\|f\|_\infty.
\]
\end{definition}

Many bounded domains enjoy this property. In particular, if $\Omega$ is a
bounded linearly convex domain with $C^{1+\varepsilon}$ boundary, then
$\Omega$ has the bounded $H^\infty$-Gleason property at every point; see
Lemmers--Wiegerinck
\cite[Theorem~4 and its proof in the arXiv version]{lemmers2002solving}.
On convex domains with this boundary regularity, their division operators
are given by the straight-segment formula
\begin{equation}\label{eq:integral-division}
  g_i(z)=\int_0^1 \frac{\partial f}{\partial z_i}\bigl(z_0+t(z-z_0)\bigr)\,\dd t,
  \qquad i=1,\ldots,m
\end{equation}
The open unit ball in $\C^m$ is a basic example. On a general linearly convex
domain, the construction instead uses suitable paths in complex-line
sections and extension from polynomials; the segments in
\eqref{eq:integral-division} need not lie in the domain.
The boundary regularity is used to obtain bounded division coefficients;
even on convex domains with merely $C^1$ boundary, the coefficients defined
by \eqref{eq:integral-division} can be unbounded, as shown by Grang\'e's
example discussed in \cite[Section~6 of the arXiv version]{lemmers2002solving}.

\subsection{NC sets, operator balls, and commuting loci}

Let
\[
  \M^d:=\bigsqcup_{n\geq 1}\M_n(\C)^d
\]
denote the graded universe of $d$-tuples of matrices. If
$X\in\M_n(\C)^d$ and $Y\in\M_m(\C)^d$, then
$X\oplus Y\in\M_{n+m}(\C)^d$ denotes the coordinatewise direct sum. If
$S\in\M_n(\C)$ is invertible, write
$S^{-1}XS=(S^{-1}X_1S,\ldots,S^{-1}X_dS)$.

\begin{definition}\label{def:nc-set-map}
An \emph{NC set} is a graded subset $\Omega=\bigsqcup_{n\geq 1}\Omega(n)$ of
$\M^d$ that is closed under direct sums. An \emph{NC domain} is an NC set whose
levels are open. A map $\Phi:\Omega\to\Lambda$ between NC sets is an
\emph{NC map} if it acts levelwise and satisfies
\[
  \Phi(X\oplus Y)=\Phi(X)\oplus\Phi(Y)
\]
and
\[
  \Phi(S^{-1}XS)=S^{-1}\Phi(X)S
\]
whenever the tuples involved belong to the corresponding levels. An
\emph{NC automorphism} is a bijective NC map whose inverse is also an NC map.
\end{definition}

This is the standard notion of NC map in free function theory;
see, for example, \cite{kaliuzhnyi2014foundations,sampatshalit2025classification,sampatshalit2025weakstar}.

\begin{definition}\label{def:DQ}
A \emph{linear pencil} in $d$ variables is a linear matrix-valued map
\[
  Q(z)=A_1z_1+\cdots+A_dz_d,
  \qquad A_1,\ldots,A_d\in\M_{r,c}(\C).
\]
For the NC operator balls considered in this paper, we assume that
$A_1,\ldots,A_d$ are linearly independent. For a tuple $X\in\M_n(\C)^d$ we set
\[
  Q(X):=\sum_{i=1}^d A_i\otimes X_i,
\]
an $rn\times cn$ block matrix. The associated \emph{NC operator ball} is the
graded set
\[
  \D_Q(n):=\{X\in\M_n(\C)^d:\|Q(X)\|<1\},\qquad n\geq 1.
\]
\end{definition}

All the domains considered in this paper arise this way, and linearity of $Q$
is used throughout; see the discussion in the introduction.

\begin{remark}\label{rem:pencil-basics}
Write $E_Q:=\operatorname{span}\{A_1,\ldots,A_d\}\subset\M_{r,c}(\C)$, equipped
with the operator space structure inherited from $\M_{r,c}(\C)$. The linear
independence of the coefficients makes $X\mapsto Q(X)$ injective at every
matrix level: if $Q(X)=0$, then
$\sum_i A_i(X_i)_{jk}=0$ for every matrix entry $(j,k)$, so all $X_i$ vanish.
Thus $X\mapsto\|Q(X)\|$ is a norm on $\M_n(\C)^d$ for every $n$, and $Q$
identifies $\D_Q(n)$ with the open unit ball of $\M_n(E_Q)$. In particular,
each level is a bounded, open, convex, balanced neighbourhood of the origin.
Moreover $Q(X\oplus Y)$ is unitarily
equivalent to $Q(X)\oplus Q(Y)$, whence
\[
  \|Q(X\oplus Y)\|=\max\{\|Q(X)\|,\|Q(Y)\|\},
\]
so $\D_Q$ is closed under direct sums and is an NC domain in the sense of
Definition~\ref{def:nc-set-map}.

The independence convention is equivalent to boundedness of the scalar
level: if the coefficients were dependent, there would be a nonzero $z$ with
$Q(z)=0$, and $\D_Q(1)$ would contain the entire line $\C z$.
\end{remark}

\begin{definition}\label{def:commuting-locus}
For a graded set $\Omega\subset\M^d$, its \emph{commuting locus} is
\[
  \Comm(\Omega)(n)
  :=
  \{X=(X_1,\ldots,X_d)\in\Omega(n):X_iX_j=X_jX_i\ \text{for all }i,j\}.
\]
\end{definition}

\begin{remark}\label{rem:classical-not-nc-Hinfty}
For an NC operator ball $\Omega=\D_Q$, $H^\infty(\Omega(1))$ denotes the
classical Banach algebra of bounded holomorphic functions on the scalar level,
not the NC function algebra
$H^\infty(\D_Q)$. We evaluate such classical functions only on commuting tuples,
using the Taylor functional calculus when the joint spectrum is contained in
the scalar domain \cite{muller2026taylor}.
\end{remark}

\begin{lemma}\label{lem:nc-map-holomorphic-fc}
Let $\Omega=\D_Q$ be an NC operator ball, and let $\Phi:\Omega\to\M^d$ be a
locally bounded NC map. Then the scalar-level restriction
$\varphi:=\Phi\big|_{\Omega(1)}$ is holomorphic on $\Omega(1)$, and for every
commuting tuple $T\in\Omega(n)$,
\[
  \Phi_n(T)=\varphi(T),
\]
where the right-hand side is given by holomorphic functional calculus. In
particular, $\Phi$ sends commuting tuples to commuting tuples.
\end{lemma}

\begin{proof}
Locally bounded NC functions on NC operator balls admit a free power series
expansion at the origin which converges uniformly on every strictly smaller
concentric ball \cite{kaliuzhnyi2014foundations}. A commuting tuple
$T\in\Omega(n)$ satisfies $\|Q(T)\|<1$, hence lies in some smaller concentric
ball, where the series may be evaluated. Restricting the series to commuting
variables abelianizes the words, and the resulting commutative series is the
Taylor series of $\varphi$. Evaluating that series on $T$ is the holomorphic
functional calculus of $\varphi$. Since the coordinates of $\varphi(T)$
belong to the commutative algebra supplied by the holomorphic functional
calculus of $T$, they commute. Thus $\Phi$ sends commuting tuples to commuting
tuples.
\end{proof}

The following lemma shows that the scalar coordinate bounds hold uniformly
at every matrix level.

\begin{lemma}\label{lem:matrix-coefficients}
Let $\Omega=\D_Q$ be an NC operator ball, let $T\in\Omega(n)$, and let
$\xi,\eta\in\C^n$ be unit vectors. Then
\[
  \bigl(\langle T_1\xi,\eta\rangle,\ldots,\langle T_d\xi,\eta\rangle\bigr)
  \in\Omega(1).
\]
Consequently, setting
\[
  r_i:=\sup_{z\in\Omega(1)}|z_i|<\infty,
  \qquad
  \rho_{\Omega}:=\sum_{i=1}^d r_i,
\]
we have $\|T_i\|\leq r_i$ for every $n\in\N$, every $T\in\Omega(n)$ and every
$i=1,\ldots,d$.
\end{lemma}

\begin{proof}
Write $z_i:=\langle T_i\xi,\eta\rangle$ and let $V_\xi:\C\to\C^n$,
$V_\eta:\C\to\C^n$ be the isometries $1\mapsto\xi$ and $1\mapsto\eta$. Viewing
$Q(T)=\sum_i A_i\otimes T_i$ as an operator from $\C^c\otimes\C^n$ to
$\C^r\otimes\C^n$, we have
\[
  (I_r\otimes V_\eta)^*\,Q(T)\,(I_c\otimes V_\xi)
  =\sum_{i=1}^d A_i\otimes\langle T_i\xi,\eta\rangle
  =Q(z),
\]
because $V_\eta^*T_iV_\xi=\langle T_i\xi,\eta\rangle$. Compressing by
isometries does not increase the norm, so $\|Q(z)\|\leq\|Q(T)\|<1$ and
$z\in\D_Q(1)=\Omega(1)$.

For the second assertion, fix $i$. For all unit vectors $\xi,\eta$ the first
part gives $|\langle T_i\xi,\eta\rangle|\leq r_i$, and taking the supremum over
$\xi$ and $\eta$ yields $\|T_i\|\leq r_i$.
\end{proof}

\begin{remark}\label{rem:linearity-needed}
Linearity is essential here. For a general matrix of NC polynomials $Q$ the
first assertion of Lemma~\ref{lem:matrix-coefficients} need not hold, since
evaluation of a monomial does not commute with compression:
$\langle T_1T_2\xi,\eta\rangle$ is unrelated to
$\langle T_1\xi,\eta\rangle\langle T_2\xi,\eta\rangle$. Compressing a single
vector ($\xi=\eta$) would only bound the numerical radius of $T_i$, and hence
$\|T_i\|$ up to a factor of~$2$; using two vectors gives the sharp bound.
\end{remark}

\begin{proposition}[Compression stability and spectra]
\label{prop:automatic-DQ-structure}
Let $\Omega=\D_Q$ be an NC operator ball.
\begin{enumerate}
\item If $T\in\Omega(n)$ and $V:\C^m\to\C^n$ is an isometry, then
$V^*TV:=(V^*T_1V,\ldots,V^*T_dV)$ belongs to $\Omega(m)$. In particular, if
with respect to a decomposition
$\C^n=\C^k\oplus\C^{n-k}$ every coordinate has block form
\[
  T_i=
  \begin{bmatrix}
    * & *\\
    0 & B_i
  \end{bmatrix},
  \qquad i=1,\ldots,d,
\]
then the lower-right corner $B=(B_1,\ldots,B_d)$ belongs to
$\Omega(n-k)$, and $B$ is commuting whenever $T$ is.

\item If $T\in\Omega(n)$ is a commuting tuple, then its Taylor joint spectrum
satisfies
\[
  \sigma(T)\subset \Omega(1).
\]
\end{enumerate}
\end{proposition}

\begin{proof}
For (1), linearity gives
\[
  (I_r\otimes V)^*\,Q(T)\,(I_c\otimes V)
  =\sum_{i=1}^d A_i\otimes V^*T_iV
  =Q(V^*TV),
\]
so $\|Q(V^*TV)\|\leq\|Q(T)\|<1$ and $V^*TV\in\Omega(m)$. If $V$ is the
inclusion of the second summand of $\C^n=\C^k\oplus\C^{n-k}$ and the
coordinates are block upper triangular as displayed, then $V^*T_iV=B_i$, which
gives the corner statement. Block upper triangular matrices multiply
blockwise, so $B_iB_j=B_jB_i$ whenever $T_iT_j=T_jT_i$.

For (2), simultaneously upper triangularize the commuting matrices
$T_1,\ldots,T_d$ \cite[Theorem~2.3.3]{horn2013matrix}. Each Taylor joint
spectral point $\lambda$ is one of the joint diagonal tuples, that is,
$\lambda_i=\langle T_ie_k,e_k\rangle$ for one of the common basis vectors
$e_k$. By Lemma~\ref{lem:matrix-coefficients} this tuple lies in $\Omega(1)$,
so $\sigma(T)\subset\Omega(1)$. The same spectral
containment also follows from \cite[Theorem~3.2]{scherer2026spectral}, which
identifies the joint spectral radius of a commuting tuple, relative to the
operator space determined by~$Q$, with the maximum of the underlying norm on
the joint spectrum.
\end{proof}

\begin{remark}\label{rem:spectral-not-enough}
The hereditary-spectrum property in
Proposition~\ref{prop:automatic-DQ-structure} does not control off-diagonal
nilpotent parts of triangular matrices. For example, with $\Omega(1)=\D$ and
\[
  T=\begin{bmatrix}0 & t \\ 0 & 0\end{bmatrix},
\]
one has $\sigma(T)=\{0\}\subset\D$ for every $t\in\C$, yet
$\|f(T)\|$ can be made arbitrarily large for $\|f\|_\infty\leq 1$ by choosing
$|t|$ large. This example is compatible with spectral containment, but it
cannot occur in an NC operator ball: such a ball
forces a uniform bound on $\|T_i\|$ at all matrix sizes, by
Lemma~\ref{lem:matrix-coefficients}. The coordinate-norm control
used in the triangular estimates below therefore needs no separate hypothesis.
\end{remark}

\subsection{Cartan domains of type I}

\begin{definition}\label{def:cartan-nc-ball}
For $p,q\geq 1$, the Cartan domain of type~I is
\[
  R_{p,q}:=\bigl\{Z\in\M_{p,q}(\C):\|Z\|<1\bigr\}\subset\C^{pq}.
\]
The associated NC operator ball is the graded set
$\cR_{p,q}=\bigsqcup_{n\geq 1}\cR_{p,q}(n)$ defined as follows. For each $n$,
write a tuple $X=(X_{ij})_{1\leq i\leq p,\,1\leq j\leq q}$ with
$X_{ij}\in\M_n(\C)$ and define
\[
  \cR_{p,q}(n):=
  \Bigl\{
    X \in \M_n(\C)^{pq}
    \;\Big|\;
    \bigl\|\bigl(X_{ij}\bigr)_{i,j}\bigr\|_{\M_{pn,qn}}<1
  \Bigr\},
\]
where the norm is the operator norm of the block matrix with $n\times n$ blocks.
Equivalently, $\cR_{p,q}=\D_Q$ for the linear pencil
\[
  Q(X)=\sum_{i=1}^p\sum_{j=1}^q E_{ij}\otimes X_{ij},
  \qquad E_{ij}:=e_ie_j^{\,*}\in\M_{p,q}(\C),
\]
that is, $Q(X)=(X_{ij})_{i,j}$ viewed as a $p\times q$ block operator matrix.
The coefficients $E_{ij}$ are linearly independent, so the scalar level
$\cR_{p,q}(1)=R_{p,q}$ is bounded.
\end{definition}

\begin{proposition}\label{prop:cartan-nc-automorphisms}
Every coordinate function on $R_{p,q}$ has supremum $1$, so that
$\rho_{\cR_{p,q}}=pq$. Moreover, there is a subgroup
$G_{p,q}\subseteq\Aut(R_{p,q})$ acting transitively on $R_{p,q}$ whose members
admit NC lifts to automorphisms of $\cR_{p,q}$.
\end{proposition}

\begin{proof}
A matrix whose only nonzero entry is $z$ in position $(i,j)$ has norm $|z|$, so
it lies in $R_{p,q}$ exactly when $|z|<1$; conversely, each entry of a strict
contraction has modulus at most $1$. Hence
$\sup_{Z\in R_{p,q}}|Z_{ij}|=1$ for all $i,j$, and summing over the $pq$
coordinates gives $\rho_{\cR_{p,q}}=pq$. McCarthy and Timoney
\cite{mccarthy2016nc-automorphisms} show that, for every $a\in R_{p,q}$,
there is a scalar automorphism $h_a$ with $h_a(0)=a$ admitting a holomorphic
NC automorphism lift. Let $G_{p,q}$ be the subgroup generated by these
automorphisms. It acts transitively, and its members lift by composing the
NC lifts and their inverses. The lifts are holomorphic NC maps, so
Lemma~\ref{lem:nc-map-holomorphic-fc} applies. When $p=q>1$, transpose is a
scalar automorphism without an NC lift; the liftable subgroup need not be
the full automorphism group.
\end{proof}

\begin{proposition}\label{prop:row-column-Gleason}
The row and column balls $R_{1,q}$ and $R_{p,1}$ have the bounded
$H^\infty$-Gleason property at~$0$, with the integral division formula
\eqref{eq:integral-division}.
\end{proposition}

\begin{proof}
The domains $R_{1,q}$ and $R_{p,1}$ are Euclidean unit balls in $\C^q$ and
$\C^p$, respectively. For the unit ball the bounded $H^\infty$-Gleason
property is classical; see Rudin
\cite[Sections~6.6.2--6.6.3, pp.~115--118]{rudin1980function}, where the
Ahern--Schneider solution at the origin is identified with
\eqref{eq:integral-division}. The uniform bound for these division operators
also follows from Lemmers--Wiegerinck
\cite[Theorem~4 and its proof in the arXiv version]{lemmers2002solving}, since
these balls are convex domains with smooth boundary.
\end{proof}

\section{A recipe for von Neumann type inequalities}\label{sec:recipe}

Throughout this section, $\Omega=\D_Q$ denotes an NC operator ball. By
Proposition~\ref{prop:automatic-DQ-structure}, spectral containment for
commuting tuples and closure under isometric compressions are automatic. By
Lemma~\ref{lem:matrix-coefficients}, the coordinate bounds
$r_i=\sup_{z\in\Omega(1)}|z_i|$ of the scalar level control the coordinate
norms at every matrix level; we keep the notation
$\rho_{\Omega}=\sum_{i=1}^d r_i$ of that lemma. We impose the following
two additional assumptions.

\medskip
\noindent\textbf{(A1) Transitive automorphisms.}
There is a subgroup $G\subseteq\Aut(\Omega(1))$ acting transitively on
$\Omega(1)$, such that every $\theta\in G$ extends to an NC automorphism
$\Theta$ of $\Omega$.

\medskip
\noindent\textbf{(A2) Gleason at the origin.}
The domain $\Omega(1)$ has the bounded $H^\infty$-Gleason property at $0$ with
constant $K_{\Omega}>0$.

The NC automorphisms in \textup{(A1)} are locally bounded: by
Lemma~\ref{lem:matrix-coefficients}, their images satisfy the coordinate
bounds $\|\Theta_n(X)_i\|\leq r_i$ for every $X\in\Omega(n)$ and
$i=1,\ldots,d$, uniformly in $n$. Lemma~\ref{lem:nc-map-holomorphic-fc}
therefore gives agreement with holomorphic functional calculus on commuting
tuples. Since these maps take values in $\Omega$, they preserve its commuting
locus.

Since $\Omega(1)$ is bounded, the coordinate functions lie in
$H^\infty(\Omega(1))$. Applying the Gleason decomposition to
$f(z)=z_j$ gives $z_j=\sum_{i=1}^d z_i g_i(z)$. Differentiating at $0$ yields
$g_j(0)=1$, so
\[
  1\leq \|g_j\|_{H^\infty(\Omega(1))}
  \leq K_{\Omega}\|z_j\|_{H^\infty(\Omega(1))}
  =K_{\Omega}\,r_j
  \leq K_{\Omega}\,\rho_{\Omega}.
\]
Thus $\rho_{\Omega}K_{\Omega}\geq 1$ always, and this lower bound is all that
the proof below will need. Note also that the product
$\rho_{\Omega}K_{\Omega}$ is unchanged if $\Omega$ is rescaled, since replacing
$\Omega$ by $\lambda\Omega$ multiplies each $r_i$ by $|\lambda|$ and divides
the Gleason constant by $|\lambda|$; the constant of
Theorem~\ref{thm:scaled-von-neumann-recipe} is therefore scale invariant, as it
must be.

\begin{theorem}\label{thm:scaled-von-neumann-recipe}
Let $\Omega=\D_Q$ be an NC operator ball.
Assume \textup{(A1)} and \textup{(A2)}.
Then for every
commuting tuple $T\in\Omega(n)$ and every $f\in H^\infty(\Omega(1))$,
\[
  \|f(T)\|\leq C_{\Omega,n}\,\|f\|_{H^\infty(\Omega(1))},
\]
where $C_{\Omega,1}=1$ and, for $n\geq 2$,
\[
  C_{\Omega,n}=(2\rho_{\Omega}K_{\Omega})^{n-1}.
\]
In particular, if every coordinate function has supremum at most $1$ on
$\Omega(1)$, then $\rho_{\Omega}\leq d$ and
$C_{\Omega,n}\leq(2dK_{\Omega})^{n-1}$.
\end{theorem}

\begin{proof}
The proof follows the inductive strategy of Hartz--Richter--Shalit
\cite{hartzrichtershalit2021row}.

By Proposition~\ref{prop:automatic-DQ-structure}, the Taylor functional
calculus defines $f(T)$ for commuting $T\in\Omega(n)$ with
$\sigma(T)\subset\Omega(1)$.

We proceed by induction on $n$. The case $n=1$ is immediate: if
$T=\lambda\in\Omega(1)$, then $f(T)=f(\lambda)$ and
$|f(\lambda)|\leq\|f\|_{H^\infty(\Omega(1))}$.

Assume the theorem holds for $n-1$, and let $T=(T_1,\ldots,T_d)\in\Omega(n)$ be
commuting. By simultaneous upper triangularization, after a unitary change of
basis we may assume
\[
  T_i=\begin{bmatrix} a_i & b_i \\ 0 & B_i \end{bmatrix},
  \qquad i=1,\ldots,d,
\]
with $a=(a_1,\ldots,a_d)\in\Omega(1)$. By~(A1), there exists
$\theta\in G$ with $\theta(a)=0$. Let $\Theta$ be a corresponding
NC lift. By Lemma~\ref{lem:nc-map-holomorphic-fc}, $\Theta_n(T)=\theta(T)$ is
commuting, and the composition property of holomorphic functional calculus
gives $f(T)=(f\circ\theta^{-1})\bigl(\Theta_n(T)\bigr)$. Replacing $T$ by
$\Theta_n(T)$ and $f$ by $f\circ\theta^{-1}$ (and using invariance of
$H^\infty(\Omega(1))$ under biholomorphic self-maps), we may assume
$a_i=0$ for all~$i$. Then
\[
  T_i=\begin{bmatrix} 0 & b_i \\ 0 & B_i \end{bmatrix},
\]
and $B=(B_1,\ldots,B_d)\in\Omega(n-1)$ is commuting by
Proposition~\ref{prop:automatic-DQ-structure}.

Suppose $\|f\|_{H^\infty(\Omega(1))}\leq\epsilon$, with $\epsilon>0$ to be
chosen. The range of $f$ lies in $\epsilon\D$. Let $c=f(0)$ and let $\psi$ be the disk
automorphism $\psi(w)=\dfrac{c-w}{1-\overline{c}\,w}$, so that $\psi(c)=0$.
Set $h:=\psi\circ f$. For $w\in\epsilon\D$,
\[
  |\psi(w)|
  \leq \sup_{|\alpha|,|\beta|\leq\epsilon}
       \left|\frac{\alpha-\beta}{1-\overline{\alpha}\beta}\right|
  \leq \frac{2\epsilon}{1+\epsilon^2}
  \leq 2\epsilon,
\]
where the middle quantity is the pseudohyperbolic diameter of
$\epsilon\D$, so $\|h\|_{H^\infty(\Omega(1))}\leq 2\epsilon$ and $h(0)=0$.
By~(A2), there exist $h_1,\ldots,h_d\in H^\infty(\Omega(1))$ with
\[
  h(z)=\sum_{i=1}^d z_i h_i(z),
  \qquad
  \|h_i\|_{H^\infty(\Omega(1))}
  \leq K_{\Omega}\,\|h\|_{H^\infty(\Omega(1))}
  \leq 2\epsilon K_{\Omega}.
\]

Using functional calculus on the commuting tuple $T$,
\[
  h(T)=\sum_{i=1}^d T_i h_i(T)
  =\sum_{i=1}^d
  \begin{bmatrix} 0 & b_i \\ 0 & B_i \end{bmatrix}
  \begin{bmatrix} h_i(0) & * \\ 0 & h_i(B) \end{bmatrix}
  =\sum_{i=1}^d
  \begin{bmatrix} 0 & b_i h_i(B) \\ 0 & B_i h_i(B) \end{bmatrix}.
\]
By Lemma~\ref{lem:matrix-coefficients}, $\|T_i\|\leq r_i$, hence from the block
structure $|b_i|^2+B_i^*B_i\leq r_i^2\,I_{n-1}$. Therefore
\begin{align*}
  \|h(T)\|
  &\leq \sum_{i=1}^d
  \bigl\|
    \begin{bmatrix}
      0 & 0 \\
      0 & h_i(B)^*(|b_i|^2+B_i^*B_i)h_i(B)
    \end{bmatrix}
  \bigr\|^{1/2} \\
  &\leq \sum_{i=1}^d r_i\,\|h_i(B)\|
  \leq \rho_{\Omega}\,C_{\Omega,n-1}\max_i\|h_i\|_{H^\infty(\Omega(1))}
  \leq 2\epsilon \rho_{\Omega}\,C_{\Omega,n-1}K_{\Omega}.
\end{align*}
Choose
$\epsilon=\bigl(2\rho_{\Omega}C_{\Omega,n-1}K_{\Omega}\bigr)^{-1}$. Since
$\rho_{\Omega}K_{\Omega}\geq 1$ and $C_{\Omega,n-1}\geq 1$, we have
$\epsilon\leq 1/2$, so
the preceding normalization takes place inside the unit disc. The choice also gives
$\|h(T)\|\leq 1$. Thus $h(T)$ is a contraction. Since $\psi^{-1}$ is a
holomorphic self-map of $\D$, the classical one-variable von Neumann inequality
applied to the contraction $h(T)$ gives $\|\psi^{-1}(h(T))\|\leq 1$. The
composition property of the Taylor functional calculus says that, whenever
$\varphi$ is holomorphic near the spectrum of $f(T)$,
$(\varphi\circ f)(T)=\varphi(f(T))$. Applying this to
$h=\psi\circ f$ and then to $\psi^{-1}$ gives
\[
  \psi^{-1}(h(T))=(\psi^{-1}\circ h)(T)=f(T).
\]
Hence $\|f(T)\|\leq 1$ for this normalized~$f$.

For general $f$, set
$g=\epsilon f/\|f\|_{H^\infty(\Omega(1))}$. Then
$\|g\|_{H^\infty(\Omega(1))}=\epsilon$ and $\|g(T)\|\leq 1$, so
\[
  \|f(T)\|\leq \epsilon^{-1}\|f\|_{H^\infty(\Omega(1))}
  =C_{\Omega,n}\,\|f\|_{H^\infty(\Omega(1))},
  \qquad
  C_{\Omega,n}=\epsilon^{-1}
  =2\rho_{\Omega}K_{\Omega}C_{\Omega,n-1}.
\]
Together with $C_{\Omega,1}=1$, this recursion gives
$C_{\Omega,n}=(2\rho_{\Omega}K_{\Omega})^{n-1}$, as claimed. The final
assertion of the theorem is immediate, since $r_i\leq 1$ for every $i$ forces
$\rho_{\Omega}\leq d$.
\end{proof}

\begin{remark}[Constants and lower bounds]
The $n$-dependence of $C_{\Omega,n}$ is not merely an artifact of the proof;
see Remark~\ref{rem:lower-bounds-row-column} below. The explicit exponential
upper bound is not claimed to be sharp.
\end{remark}

\subsection{Finite-dimensional representations}
\label{subsec:fd-representations}

Write $A(\Omega(1))$ for the uniform algebra of functions holomorphic on
$\Omega(1)$ and continuous on $\overline{\Omega(1)}$, equipped with the scalar
supremum norm
\[
  \|f\|_\infty=\sup_{z\in\Omega(1)}|f(z)|.
\]
This is the classical ball-algebra analogue on the scalar level, as in
Remark~\ref{rem:classical-not-nc-Hinfty}; in particular the norm does not see
matrix levels of~$\Omega$.

In this subsection we read Theorem~\ref{thm:scaled-von-neumann-recipe} as a
statement about the bounded unital representations of $A(\Omega(1))$ on
$\C^n$, and show that, up to similarity, these are exactly the evaluations at
commuting tuples in $\overline{\Omega}(n)$. Two known results frame the
question. In the NC function algebra $A(\D_Q)$ the norm already sees every
matrix level, and the coordinate evaluation map identifies the completely
contractive $n$-dimensional representations of $A(\D_Q)$ with
$\overline{\D_Q(n)}$ \cite[Proposition~5.1]{sampatshalit2025classification}.
For the ball algebra $A(\Ball_d)$, whose norm is the scalar one,
Hartz--Richter--Shalit used their von Neumann type inequality to show that the
bounded unital $n$-dimensional representations are precisely the maps
$p\mapsto p(T)$ with $T$ a commuting tuple jointly similar to a row
contraction, and obtained analogous results for the multiplier algebras of the
weighted Dirichlet spaces on $\Ball_d$
\cite[Section~6 and Theorem~6.4]{hartzrichtershalit2021row}.
Theorem~\ref{thm:fd-rep-classification} below extends the description for
$A(\Ball_d)$ to every NC operator ball satisfying \textup{(A1)} and
\textup{(A2)}.

We begin with the evaluations themselves. Since the scalar norm does not see
the matrix levels of $\Omega$, their boundedness is not tautological; it is
precisely the content of Theorem~\ref{thm:scaled-von-neumann-recipe}.

\begin{corollary}[Evaluations on the closed commuting locus]
\label{cor:closed-locus}
Let $\Omega=\D_Q$ be an NC operator ball, and assume
\textup{(A1)} and \textup{(A2)}.
Let $C_{\Omega,n}$ be the constant of
Theorem~\ref{thm:scaled-von-neumann-recipe}. Then for every commuting tuple
$T\in\overline{\Omega}(n)$,
\[
  \|p(T)\|\leq C_{\Omega,n}\,\|p\|_\infty
  \qquad\text{for every polynomial }p,
\]
and $p\mapsto p(T)$ extends uniquely to a bounded unital algebra homomorphism
$\Phi_T:A(\Omega(1))\to\M_n(\C)$ with $\|\Phi_T\|\leq C_{\Omega,n}$. If
moreover $T\in\Omega(n)$, then for every $f\in A(\Omega(1))$, the value
$\Phi_T(f)=f(T)$ is given by the Taylor functional calculus.
\end{corollary}

\begin{proof}
Every level of $\Omega$ is balanced by Remark~\ref{rem:pencil-basics}, so for
$0<r<1$ we have $\|Q(rT)\|=r\|Q(T)\|\leq r<1$, and $rT$ is a commuting tuple in
$\Omega(n)$. Theorem~\ref{thm:scaled-von-neumann-recipe} gives
$\|p(rT)\|\leq C_{\Omega,n}\|p\|_\infty$, and letting $r\to1$ yields the
displayed estimate, since $p(rT)\to p(T)$.

For the extension it suffices to know that polynomials are dense in
$A(\Omega(1))$. Let $f\in A(\Omega(1))$ and $0<r<1$. As $f$ is continuous on the
compact set $\overline{\Omega(1)}$, it is uniformly continuous there, so
$f(r\,\cdot\,)\to f$ uniformly on $\overline{\Omega(1)}$ as $r\to1$. Moreover
$\|Q(rz)\|\leq r<1$ for $z\in\overline{\Omega(1)}$, so
$\overline{\Omega(1)}\subset r^{-1}\Omega(1)$ and $f(r\,\cdot\,)$ is holomorphic
on the balanced domain $r^{-1}\Omega(1)$; its homogeneous expansion therefore
converges uniformly on the compact subset $\overline{\Omega(1)}$. Hence $f$ is a
uniform limit of polynomials, and $\Phi_T$ is the unique continuous extension of
$p\mapsto p(T)$.

Finally, let $T\in\Omega(n)$. By
Proposition~\ref{prop:automatic-DQ-structure}\textup{(2)},
$\sigma(T)\subset\Omega(1)$, so the Taylor functional calculus of
Remark~\ref{rem:classical-not-nc-Hinfty} defines $f(T)$ for every
$f\in A(\Omega(1))\subset H^\infty(\Omega(1))$, and $f\mapsto f(T)$ is a
unital homomorphism. It is continuous for locally uniform convergence on
$\Omega(1)$, hence for the supremum norm, and it agrees with $\Phi_T$ on
polynomials; by density, $\Phi_T(f)=f(T)$.
\end{proof}

For commuting $T\in\overline{\Omega}(n)$, the corollary defines $\Phi_T$ on
$A(\Omega(1))$ by uniform polynomial approximation. When
$\sigma(T)\subset\Omega(1)$, this evaluation agrees with the Taylor
functional calculus. In general, however, the joint spectrum may meet the
boundary of $\Omega(1)$, and a function in $H^\infty(\Omega(1))$ need not
extend holomorphically to a neighbourhood of that spectrum. The corollary
therefore does not assert an $H^\infty(\Omega(1))$ functional calculus for
every tuple in the closed commuting locus.

The representations of Corollary~\ref{cor:closed-locus} do not exhaust the
bounded representations of $A(\Omega(1))$, because the two sides have
different symmetries. If $\Phi$ is a bounded unital representation of
$A(\Omega(1))$ on $\C^n$ and $S\in\M_n(\C)$ is invertible, then
$S^{-1}\Phi(\,\cdot\,)S$ is another one, of norm at most
$\|S\|\,\|S^{-1}\|\,\|\Phi\|$. The commuting locus, by contrast, is only
unitarily invariant, since $\Omega(n)$ is the open unit ball of the norm
$X\mapsto\|Q(X)\|$ on $\M_n(E_Q)$ (Remark~\ref{rem:pencil-basics}). The
natural candidate for the set of all bounded representations is therefore the
similarity orbit of the closed commuting locus.

To identify it, we compare $A(\Omega(1))$ with an operator algebra whose
completely contractive representations can be read off from the pencil,
namely the algebra $A(\mathfrak{C}\D_Q)$ of NC functions on the commuting locus
$\mathfrak{C}\D_Q:=\Comm(\D_Q)$, regarded as the commutative free subvariety of
$\D_Q$ \cite[Example~6.13]{sampatshalit2025classification}. Following
\cite[Section~1.2]{sampatshalit2025classification}, $A(\mathfrak{C}\D_Q)$ is
the closure of the free polynomials in the graded norm
$\|f\|=\sup_n\sup_{X\in\Comm(\Omega)(n)}\|f(X)\|$, and it is a unital operator
algebra with the matrix norms
$\|[f_{jk}]\|=\sup_n\sup_{X\in\Comm(\Omega)(n)}\|[f_{jk}(X)]\|$. For the row
ball, $A(\mathfrak{C}\D_Q)$ identifies with the closure of the polynomials in
the multiplier norm of $\operatorname{Mult}(H^2_d)$. Its scalar-level image
is contained in the ball algebra $A(\Ball_d)$, and this inclusion is proper
when $d\geq 2$; see
\cite[Section~1]{sampatshalit2025classification} and \cite{hartz2025operator}.

In the following theorem $z_1,\ldots,z_d$ denote the coordinate functions in
$A(\Omega(1))$ and $Z_1,\ldots,Z_d$ the coordinate functions in
$A(\mathfrak{C}\D_Q)$. Since polynomials are dense in both algebras (for
$A(\Omega(1))$ by the proof of Corollary~\ref{cor:closed-locus}, and for
$A(\mathfrak{C}\D_Q)$ by definition), a bounded unital homomorphism on either
algebra is determined by its values on the coordinates.

\begin{theorem}\label{thm:fd-rep-classification}
Let $\Omega=\D_Q$ be an NC operator ball, and assume
\textup{(A1)} and \textup{(A2)}. For $T=(T_1,\ldots,T_d)\in\M_n(\C)^d$ the following are
equivalent:
\begin{enumerate}
\item there is a bounded unital homomorphism
$\Phi:A(\Omega(1))\to\M_n(\C)$ with $\Phi(z_i)=T_i$ for $i=1,\ldots,d$;
\item there is a bounded unital homomorphism
$\rho:A(\mathfrak{C}\D_Q)\to\M_n(\C)$ with $\rho(Z_i)=T_i$ for
$i=1,\ldots,d$;
\item $T$ is jointly similar to a commuting tuple in $\overline{\Omega}(n)$,
that is, $T=S^{-1}XS$ for some invertible $S\in\M_n(\C)$ and some commuting
$X\in\overline{\Omega}(n)$.
\end{enumerate}
In this case $S$ in \textup{(3)} may be chosen with
$\|S\|\,\|S^{-1}\|\leq\|\rho\|_{\mathrm{cb}}$, and then
$\|\Phi\|\leq C_{\Omega,n}\|S\|\,\|S^{-1}\|$, where $C_{\Omega,n}$ is the
constant of Theorem~\ref{thm:scaled-von-neumann-recipe}. In particular, the
bounded unital $n$-dimensional representations of $A(\Omega(1))$ and of
$A(\mathfrak{C}\D_Q)$ are parametrized by the same set of tuples, namely the
similarity orbit of the closed commuting locus $\Comm(\overline{\Omega})(n)$.
\end{theorem}

\begin{proof}
\textup{(3)}$\Rightarrow$\textup{(1)}. Let $T=S^{-1}XS$ with $X$ as in
\textup{(3)}. By Corollary~\ref{cor:closed-locus}, $X$ induces a bounded unital
homomorphism $\Phi_X:A(\Omega(1))\to\M_n(\C)$ with $\Phi_X(z_i)=X_i$ and
$\|\Phi_X\|\leq C_{\Omega,n}$. Then $\Phi:=S^{-1}\Phi_X(\,\cdot\,)S$ is a
bounded unital homomorphism with $\Phi(z_i)=T_i$ and
$\|\Phi\|\leq C_{\Omega,n}\|S\|\,\|S^{-1}\|$.

\textup{(1)}$\Rightarrow$\textup{(2)}. Since
$\mathfrak{C}\D_Q(1)=\Omega(1)$, restriction to the scalar level defines a
contractive unital homomorphism
\[
  R:A(\mathfrak{C}\D_Q)\longrightarrow A(\Omega(1)),
  \qquad R(F)=F\big|_{\Omega(1)}.
\]
Indeed, the graded norm dominates the scalar supremum norm, so polynomial
approximation in $A(\mathfrak{C}\D_Q)$ gives uniform polynomial approximation
on $\overline{\Omega(1)}$. Thus $R$ takes values in $A(\Omega(1))$ and has
norm at most $1$. Set $\rho:=\Phi\circ R$. Then $\rho$ is a bounded unital
homomorphism with $\rho(Z_i)=\Phi(z_i)=T_i$ and
$\|\rho\|\leq\|\Phi\|$, as required.

\textup{(2)}$\Rightarrow$\textup{(3)}. Since $\rho$ takes values in
$\M_n(\C)$, it is completely bounded by Smith's lemma
\cite{smith1983completely}, with $\|\rho\|_{\mathrm{cb}}=\|\rho^{(n)}\|\leq
n\|\rho\|$. As $A(\mathfrak{C}\D_Q)$ is a unital operator algebra, Paulsen's
similarity theorem \cite[Theorem~9.1]{paulsen2002completely} provides an
invertible $S\in\M_n(\C)$ with $\|S\|\,\|S^{-1}\|\leq\|\rho\|_{\mathrm{cb}}$
such that $\sigma:=S\rho(\,\cdot\,)S^{-1}$ is completely contractive. Put
$X:=STS^{-1}$, so that $X_i=\sigma(Z_i)$. The coordinates $Z_i$ commute in
$A(\mathfrak{C}\D_Q)$, since they are evaluated only on commuting tuples, and
hence so do the $X_i$. Moreover, the element
$Q(Z)=\sum_{i=1}^dA_i\otimes Z_i$ of $\M_{r,c}\bigl(A(\mathfrak{C}\D_Q)\bigr)$
has norm $\sup_n\sup_{Y\in\Comm(\Omega)(n)}\|Q(Y)\|\leq1$, and complete
contractivity of $\sigma$ gives
\[
  \|Q(X)\|
  =\bigl\|(\mathrm{id}_{\M_{r,c}}\otimes\sigma)\bigl(Q(Z)\bigr)\bigr\|
  \leq\|Q(Z)\|
  \leq1.
\]
Thus $X$ is a commuting tuple in $\overline{\Omega}(n)$ and $T=S^{-1}XS$.
This is the argument of \cite[Proposition~5.1]{sampatshalit2025classification}
applied to $A(\mathfrak{C}\D_Q)$; compare the identification of the completely
contractive representations of $A(\mathfrak{V})$ for homogeneous subvarieties
$\mathfrak{V}$ in \cite[Proposition~5.3]{sampatshalit2025classification}.

The norm estimates follow by combining the bound on $\|S\|\,\|S^{-1}\|$ from
the last step with the estimate for $\|\Phi\|$ from the first.
\end{proof}

\begin{remark}\label{rem:fd-rep-hypotheses}
Only the implication \textup{(3)}$\Rightarrow$\textup{(1)} uses \textup{(A1)}
and \textup{(A2)}. Indeed, a commuting $X\in\overline{\Omega}(n)$ satisfies
$\|p(X)\|=\lim_{r\to1}\|p(rX)\|\leq\|p\|_{A(\mathfrak{C}\D_Q)}$, so it always
induces a completely contractive representation of $A(\mathfrak{C}\D_Q)$, and
\textup{(2)}$\Leftrightarrow$\textup{(3)} holds for every NC operator ball.
The content of Theorem~\ref{thm:fd-rep-classification} is therefore that
passing to the much smaller scalar norm of $A(\Omega(1))$ loses none of these
$n$-dimensional representations; this is exactly what
Theorem~\ref{thm:scaled-von-neumann-recipe} supplies.
\end{remark}

\begin{remark}[Sharpness of the description]
\label{rem:similarity-gap}
Neither the similarity in Theorem~\ref{thm:fd-rep-classification}\textup{(3)}
nor the constant $C_{\Omega,n}$ can be dispensed with.

The similarity orbit is strictly larger than the closed commuting locus. Take
$\Omega(1)=\D$, so that $\Omega(n)$ consists of the strict contractions, and
let $T$ be the nilpotent matrix of Remark~\ref{rem:spectral-not-enough}. Then
$f(T)=f(0)I+f'(0)T$ for every $f\in A(\D)$, and the Schwarz--Pick estimate
$|f'(0)|\leq\|f\|_\infty$ gives $\|\Phi_T\|\leq 1+|t|$. Thus $T$ defines a
bounded representation of $A(\D)$ for every $t\in\C$, although
$\|Q(T)\|=|t|$ places it outside $\overline{\Omega}(2)$ as soon as $|t|>1$. In
accordance with Theorem~\ref{thm:fd-rep-classification}, for $t\neq0$ we have
$T=S^{-1}XS$ with
\[
  X=\begin{bmatrix}0&1\\0&0\end{bmatrix}\in\overline{\Omega}(2),
  \qquad
  S=\begin{bmatrix}1&0\\0&t\end{bmatrix},
\]
so what controls $\|\Phi_T\|$ is the similarity constant of $T$, not the norm
of $T$ itself. Nor can similarity be relaxed to a spectral condition: for the ball
algebra, a commuting tuple $T$ with $\sigma(T)\subset\overline{\Ball_d}$ need
not define a bounded representation of $A(\Ball_d)$
\cite[Remark~6.5]{hartzrichtershalit2021row}.

On the other hand, the closed commuting locus is strictly larger than the set
of contractive representations of $A(\Omega(1))$. For the polydisc with
$d\geq3$ a commuting tuple of contractions lies in $\overline{\Omega}(n)$ and
yet need not satisfy $\|p(T)\|\leq\|p\|_\infty$ \cite{varopoulos1974inequality},
so $\Phi_T$ is not even contractive, although the corresponding
representation of $A(\mathfrak{C}\D_Q)$ is completely contractive. Thus the passage from $A(\mathfrak{C}\D_Q)$ to
$A(\Omega(1))$ preserves boundedness but not contractivity, and the constant
$C_{\Omega,n}$ in Corollary~\ref{cor:closed-locus} cannot in general be taken
equal to~$1$. This is the precise point at which the scalar-level picture
departs from the NC one, where $\overline{\D_Q(n)}$ is exactly the set of
completely contractive representations of $A(\D_Q)$.
\end{remark}

\section{Gleason's problem on product domains}\label{sec:gleason-product}

\begin{lemma}[Products of NC operator balls]
\label{lem:product-nc-ball}
Let $\Omega'=\D_{Q'}$ and $\Omega''=\D_{Q''}$ be NC operator balls, where
$Q'$ is a linear pencil in $d'$ variables with coefficients in
$\M_{r',c'}(\C)$ and $Q''$ is a linear pencil in $d''$ variables with
coefficients in $\M_{r'',c''}(\C)$. Define a linear pencil $Q$ in $d'+d''$
variables, with coefficients in $\M_{r'+r'',\,c'+c''}(\C)$, by
\[
  Q(X,Y)
  :=
  Q'(X)\oplus Q''(Y)
  :=
  \begin{pmatrix}
    Q'(X) & 0 \\
    0 & Q''(Y)
  \end{pmatrix}.
\]
Explicitly, the coefficients of $Q$ are $A_i'\oplus 0$ for $i\leq d'$ and
$0\oplus A_j''$ for $j\leq d''$, where $A_i'$ and $A_j''$ denote the
coefficients of $Q'$ and $Q''$. Then
$\|Q(X,Y)\|=\max\bigl\{\|Q'(X)\|,\|Q''(Y)\|\bigr\}$, and the levelwise
product $\Omega'\times\Omega''$ coincides with $\D_Q$. Iterating, a finite
levelwise product of NC operator balls is again an NC operator ball.
Furthermore $\Omega(1)=\Omega'(1)\times\Omega''(1)$ is bounded, and the
coordinate suprema of the product are those of the
factors, so that $\rho_{\Omega}=\rho_{\Omega'}+\rho_{\Omega''}$.
\end{lemma}

\begin{proof}
The coefficients $A_i'\oplus0$ and $0\oplus A_j''$ of $Q$ are linearly
independent, since any linear relation splits into separate relations among
the coefficients of $Q'$ and $Q''$ in the two diagonal blocks.
If $A$ and $B$ are (possibly rectangular) operators, then
$\|A\oplus B\|=\max\{\|A\|,\|B\|\}$. Thus $\|Q(X,Y)\|<1$ if and only if
$\|Q'(X)\|<1$ and $\|Q''(Y)\|<1$. The claims about boundedness and about the
coordinate suprema are immediate from the resulting product description of
$\Omega(1)$, and the last claim follows by iterating the two-factor
construction.
\end{proof}

If the factors $\Omega'$ and $\Omega''$ satisfy \textup{(A1)}, with
transitive liftable subgroups $G'$ and $G''$, then their product
$\Omega=\Omega'\times\Omega''$ also satisfies \textup{(A1)}. Indeed,
$G'\times G''$ acts componentwise and transitively on
$\Omega(1)=\Omega'(1)\times\Omega''(1)$, and the corresponding NC lifts act
componentwise on the product domain. Thus
Theorem~\ref{thm:scaled-von-neumann-recipe} would be stable under products if
assumption \textup{(A2)} were also stable.

The obstruction is the bounded $H^\infty$-Gleason property. If
$f\in H^\infty(\Omega(1))$ vanishes at the origin, one can formally write
\[
  f(z,w)=\bigl(f(z,w)-f(0,w)\bigr)+f(0,w).
\]
For each fixed $w$, the first term vanishes as a function of $z$, so the
Gleason property on $\Omega'(1)$ gives divisors in the $z$-variables. The
difficulty is that the abstract property does not guarantee that these divisors
can be chosen to depend holomorphically on the parameter $w$. Thus
\textup{(A2)} requires an additional argument for product domains.

In this section we give sufficient conditions under which this obstruction
disappears, and we apply them to products of row and column balls.

\begin{proposition}\label{prop:product-Gleason}
Let $\Omega_1\subset\C^{k_1}$ and $\Omega_2\subset\C^{k_2}$ be domains, and
let $z_0\in\Omega_1$, $w_0\in\Omega_2$. Assume that $\Omega_2$ has the bounded
$H^\infty$-Gleason property at $w_0$ with constant $K_2$. Assume also that
there is a constant $K_1>0$ such that, for every
$h\in H^\infty(\Omega_1\times\Omega_2)$ satisfying $h(z_0,w)=0$ for all
$w\in\Omega_2$, there exist $H_1,\ldots,H_{k_1}\in
H^\infty(\Omega_1\times\Omega_2)$ with
\[
  h(z,w)=\sum_{i=1}^{k_1}\bigl(z_i-(z_0)_i\bigr)H_i(z,w),
  \qquad
  \|H_i\|_\infty\leq K_1\,\|h\|_\infty.
\]
Here the divisors are jointly holomorphic, and the norm bound is uniform on
the entire product domain.
Then $\Omega_1\times\Omega_2$ has the bounded $H^\infty$-Gleason property at
$(z_0,w_0)$ with constant $\max\{2K_1,K_2\}$.
\end{proposition}

\begin{proof}
Let $f\in H^\infty(\Omega_1\times\Omega_2)$ with $f(z_0,w_0)=0$. Write
\[
  f(z,w)=\bigl[f(z,w)-f(z_0,w)\bigr]+f(z_0,w).
\]
For fixed $z_0$, the slice $w\mapsto f(z_0,w)$ lies in
$H^\infty(\Omega_2)$ and vanishes at $w_0$. By the Gleason property on
$\Omega_2$, there exist $G_1,\ldots,G_{k_2}\in H^\infty(\Omega_2)$ with
\[
  f(z_0,w)=\sum_{j=1}^{k_2}\bigl(w_j-(w_0)_j\bigr)G_j(w),
  \qquad
  \|G_j\|_\infty\leq K_2\,\|f\|_\infty.
\]
Define $\widetilde{G}_j(z,w):=G_j(w)$; then
$\widetilde{G}_j\in H^\infty(\Omega_1\times\Omega_2)$ and
\[
  f(z_0,w)=\sum_{j=1}^{k_2}\bigl(w_j-(w_0)_j\bigr)\widetilde{G}_j(z,w).
\]

Set $h(z,w):=f(z,w)-f(z_0,w)$. Then
$h\in H^\infty(\Omega_1\times\Omega_2)$, $h(z_0,w)=0$ for every
$w\in\Omega_2$, and $\|h\|_\infty\leq 2\|f\|_\infty$. By the
parameter-dependent division hypothesis, there exist
$H_1,\ldots,H_{k_1}\in H^\infty(\Omega_1\times\Omega_2)$ such that
\[
  h(z,w)=\sum_{i=1}^{k_1}\bigl(z_i-(z_0)_i\bigr)H_i(z,w),
  \qquad
  \|H_i\|_\infty\leq K_1\,\|h\|_\infty
  \leq 2K_1\,\|f\|_\infty.
\]
Therefore
\[
  f(z,w)
  =
  \sum_{i=1}^{k_1}\bigl(z_i-(z_0)_i\bigr)H_i(z,w)
  +\sum_{j=1}^{k_2}\bigl(w_j-(w_0)_j\bigr)\widetilde{G}_j(z,w).
\]
Both families of coefficients are bounded by
\[
  \max\{2K_1,K_2\}\,\|f\|_\infty.
\]
Thus $\Omega_1\times\Omega_2$ has the bounded $H^\infty$-Gleason property at
$(z_0,w_0)$.
\end{proof}

\begin{corollary}\label{cor:product-balls-Gleason}
Let
\[
  D=\Ball_{m_1}\times\cdots\times\Ball_{m_r}\subset\C^d,
  \qquad d=\sum_{j=1}^r m_j,
\]
be a finite product of Euclidean unit balls. Then $D$ has the bounded
$H^\infty$-Gleason property at~$0$.
\end{corollary}

\begin{proof}
Each Euclidean unit ball $\Ball_m$ is a row ball $R_{1,m}$, so
Proposition~\ref{prop:row-column-Gleason} gives the bounded
$H^\infty$-Gleason property with the integral division formula on each factor.
To verify the parameter-dependent hypothesis, let $U$ be a domain and let
$h\in H^\infty(\Ball_m\times U)$ satisfy $h(0,w)=0$ for all $w\in U$.
Set
\[
  H_i(z,w):=\int_0^1\frac{\partial h}{\partial z_i}(tz,w)\,\dd t,
  \qquad i=1,\ldots,m.
\]
These functions are jointly holomorphic: locally in $(z,w)$ the integration
paths lie in a compact subset of $\Ball_m\times U$, so differentiation under
the integral is valid. Applying the bound from
Proposition~\ref{prop:row-column-Gleason} to each slice $h(\cdot,w)$ gives
$\|H_i\|_\infty\leq K_{\Ball_m}\|h\|_\infty$, uniformly in $w$, and
$h(z,w)=\sum_i z_iH_i(z,w)$.
Thus Proposition~\ref{prop:product-Gleason} applies iteratively, taking a
single ball as the first factor and the product already treated as the second.
\end{proof}

\begin{remark}
When $r\geq 2$, the boundary of
$D=\Ball_{m_1}\times\cdots\times\Ball_{m_r}$ is not smooth. The proof of
Corollary~\ref{cor:product-balls-Gleason} relies instead on the explicit
integral division formula on each ball factor.
\end{remark}

\section{Products of row balls and column balls}\label{sec:row-column}

Write
\[
  R_{p,q}:=\{Z\in\M_{p,q}(\C):\|Z\|<1\}
\]
for the Cartan domain of type~I, and let $\cR_{p,q}$ be the associated NC
operator ball as in Definition~\ref{def:cartan-nc-ball}. For $p=1$, the
associated NC operator ball is the row ball, and for $q=1$ it is the column
ball.

Fix integers $r,s\geq 0$, positive integers $q_1,\ldots,q_r$ and
$p_1,\ldots,p_s$, and set
\[
  d=\sum_{j=1}^r q_j+\sum_{\ell=1}^s p_\ell.
\]
Let
\[
  \Omega
  =
  R_{1,q_1}\times\cdots\times R_{1,q_r}
  \times
  R_{p_1,1}\times\cdots\times R_{p_s,1}
  \subset\C^d
\]
be the corresponding scalar product domain.
For each $n\geq 1$, define
\[
  \cR_{\Omega}(n)
  :=
  \cR_{1,q_1}(n)\times\cdots\times\cR_{1,q_r}(n)
  \times
  \cR_{p_1,1}(n)\times\cdots\times\cR_{p_s,1}(n),
\]
viewed as a subset of $\M_n(\C)^d$ by listing all matrix coordinates in the
order induced by the product.
A point of $\cR_{\Omega}(n)$ is simply a tuple of $n\times n$ matrices grouped
according to the $r+s$ block structures. In the von Neumann type inequality
below we restrict, as in Theorem~\ref{thm:scaled-von-neumann-recipe}, to those
product tuples whose matrix coordinates commute with one another.

\begin{lemma}\label{lemma:product-hypotheses}
The product NC operator ball
$\cR_{\Omega}=\bigsqcup_{n\geq 1}\cR_{\Omega}(n)$ over $\Omega$ is an NC
operator ball $\D_Q$ satisfying \textup{(A1)} and \textup{(A2)}, and
$\rho_{\cR_{\Omega}}=d$.
\end{lemma}

\begin{proof}
By Lemma~\ref{lem:product-nc-ball}, the levelwise product of the row and
column NC operator balls is again an NC operator ball, and
its coordinate suprema are those of the factors; by
Proposition~\ref{prop:cartan-nc-automorphisms} each factor $R_{1,q_j}$ or
$R_{p_\ell,1}$ contributes $q_j$, respectively $p_\ell$, so
$\rho_{\cR_{\Omega}}=\sum_j q_j+\sum_\ell p_\ell=d$.
Moreover, property \textup{(A1)} is stable under such products, and hence
$\cR_{\Omega}$ satisfies \textup{(A1)}. Finally, at the scalar level, each
$R_{1,q}$ and $R_{p,1}$ is a Euclidean ball. Therefore $\Omega$ has the bounded
$H^\infty$-Gleason property at the origin by
Corollary~\ref{cor:product-balls-Gleason}, so \textup{(A2)} holds.
\end{proof}

\begin{corollary}\label{cor:row-column-product}
With notation as above, for every commuting tuple $T\in\cR_{\Omega}(n)$ and
every $f\in H^\infty(\Omega)$,
\[
  \|f(T)\|\leq C_{\Omega,n}\,\|f\|_{H^\infty(\Omega)},
  \qquad
  C_{\Omega,n}=(2dK_\Omega)^{n-1},
\]
where $K_\Omega$ is the Gleason constant of $\Omega$ at~$0$ from
Corollary~\ref{cor:product-balls-Gleason}.
\end{corollary}

\begin{proof}
Combine Theorem~\ref{thm:scaled-von-neumann-recipe} with
Lemma~\ref{lemma:product-hypotheses}.
\end{proof}

\begin{remark}[The polydisc case]
When all factors are one-dimensional, that is, when $q_j=p_k=1$ for all
$j,k$, the scalar domain $\Omega$ is the unit polydisc $\D^d$, and
$\cR_{\Omega}$ is the NC unit polydisc. Thus
Corollary~\ref{cor:row-column-product} gives a finite-dimensional von Neumann
type inequality for commuting $n\times n$ contractions.

This special case belongs to the classical theory of von Neumann inequalities
on the polydisc (see~\cite{sznagyfoias1970,shalit2021dilation}). For $d=1$ this
is von Neumann's inequality~\cite{vonneumann1951spektral}, and for $d=2$ it is
Ando's theorem~\cite{ando1963pair}; in both cases the optimal constant is $1$
for every $n$. For $d\geq 3$ the inequality fails with constant $1$
\cite{varopoulos1974inequality}, and no constant independent of the matrix size
$n$ is known; whether such a constant exists is a well-known open problem, to
which we return in Remark~\ref{rem:lower-bounds-row-column}. For each fixed $n$
finite constants do exist, and much sharper estimates than ours are known: work
of Daher, as presented by Pisier, gives an upper bound of order
$(\log(n+1))^d$, up to
constants depending on $d$; see
\cite[Chapter~4, Corollary~4.21]{pisier2001similarity}. The bound obtained from
Theorem~\ref{thm:scaled-von-neumann-recipe} is not sharp in the polydisc case.
The same structural argument applies uniformly to products of row and column NC
operator balls, and in those cases the resulting estimate is new.
\end{remark}

\begin{remark}[Lower bounds from row and column factors]
\label{rem:lower-bounds-row-column}
Let $C_{\Omega,n}^{\mathrm{opt}}$ denote the optimal constant in
Corollary~\ref{cor:row-column-product}. Write $C_{m,n}^{\mathrm{row}}$ for the
smallest constant such that
\[
  \|f(T)\|\leq C_{m,n}^{\mathrm{row}}\|f\|_{H^\infty(\Ball_m)}
\]
for every commuting $T\in\cR_{1,m}(n)$ and every
$f\in H^\infty(\Ball_m)$. Here $m$ is the number of variables and $n$ is the
matrix size. If at least one factor has dimension at least two, then
$C_{\Omega,n}^{\mathrm{opt}}$ cannot remain bounded as $n\to\infty$.
Indeed, suppose first that $q_j\geq 2$ for some row factor. By setting all
other factors equal to zero, $\cR_{\Omega}(n)$ contains a copy of the row ball
$\cR_{1,q_j}(n)$. Taking functions that depend only on this factor gives
$C_{\Omega,n}^{\mathrm{opt}}\geq C_{q_j,n}^{\mathrm{row}}$, so any lower-bound
example for the row ball gives the same lower bound for
$C_{\Omega,n}^{\mathrm{opt}}$.

Hartz--Richter--Shalit show that the optimal constants for the row ball in
$m\geq 2$ variables tend to infinity with the matrix size; more precisely, if
$N_k=\binom{m+k}{m}$, then their lower bounds imply
\[
  C_{m,N_k}^{\mathrm{row}}
  \geq
  (2\pi)^{\frac{m-1}{4}}m^{-\frac14}k^{\frac{m-1}{4}},
\]
see \cite[Remark~4.8]{hartzrichtershalit2021row}. Therefore, if some
$q_j\geq 2$, then $C_{\Omega,n}^{\mathrm{opt}}\to\infty$ as $n\to\infty$.
The same conclusion holds when some column factor has $p_k\geq 2$, by applying
the row-ball lower bounds after transposing the matrix tuples.

If every factor is one-dimensional, so that $\Omega$ is the unit polydisc
$\D^d$, the same conclusion is not known. For $d=1$ and $d=2$ the optimal
constants equal $1$ for every $n$, by von Neumann's inequality and Ando's
theorem. For $d\geq 3$ it is a well-known open problem whether
$C_{\Omega,n}^{\mathrm{opt}}$ remains bounded as $n\to\infty$; see
\cite[Chapter~1]{pisier2001similarity} and \cite{hartz2025polydisc}.
\end{remark}

\section{Conditional results for general Cartan domains}\label{sec:cartan}

The preceding section gives an unconditional result for products of row and
column balls. For a general type-I Cartan domain $R_{p,q}$, the associated NC
operator ball $\cR_{p,q}$ has $\rho_{\cR_{p,q}}=pq$ and satisfies
\textup{(A1)}, by Definition~\ref{def:cartan-nc-ball} and
Proposition~\ref{prop:cartan-nc-automorphisms}.

The remaining condition is the bounded $H^\infty$-Gleason property
\textup{(A2)}. This holds when $\min(p,q)=1$, by
Proposition~\ref{prop:row-column-Gleason}. For $\min(p,q)\geq 2$, we are not
aware of a proof of this property in the literature. We therefore state the
following corollary under hypothesis \textup{(A2)}.

\begin{corollary}\label{cor:cartan-conditional}
Let $p,q\geq 1$, and assume that $R_{p,q}$ has the bounded
$H^\infty$-Gleason property at~$0$ with constant $K_{p,q}$. Then for every
commuting tuple $T\in\cR_{p,q}(n)$ and every $f\in H^\infty(R_{p,q})$,
\[
  \|f(T)\|
  \leq
  C_{p,q,n}\,\|f\|_{H^\infty(R_{p,q})},
\]
where $C_{p,q,1}=1$ and, for $n\geq 2$,
\[
  C_{p,q,n}=(2pqK_{p,q})^{n-1}.
\]
\end{corollary}

\begin{proof}
The scalar domain $R_{p,q}$ lies in $\C^{pq}$, so here $d=pq$.
By Definition~\ref{def:cartan-nc-ball} and
Proposition~\ref{prop:cartan-nc-automorphisms}, the associated NC operator ball
$\cR_{p,q}$ satisfies \textup{(A1)} and has $\rho_{\cR_{p,q}}=pq$.
The assumed bounded $H^\infty$-Gleason property is precisely \textup{(A2)}.
The conclusion follows from Theorem~\ref{thm:scaled-von-neumann-recipe}.
\end{proof}

With additional hypotheses ensuring that the bounded $H^\infty$-Gleason
property passes to products, the same argument gives a von Neumann type
inequality for product domains built from general type~I Cartan factors. The
next corollary records the two-factor case; finite products can be handled
iteratively under the corresponding hypotheses.

\begin{corollary}\label{cor:cartan-product-conditional}
Let
\[
  \Omega=R_{p_1,q_1}\times R_{p_2,q_2}\subset\C^d,
  \qquad
  d=p_1q_1+p_2q_2,
\]
and let
\[
  \cR_{\Omega}(n):=\cR_{p_1,q_1}(n)\times\cR_{p_2,q_2}(n)
\]
be the corresponding product NC operator ball. Assume that both scalar factors
have the bounded $H^\infty$-Gleason property at~$0$, with constants $K_1$ and
$K_2$, and that, after possibly interchanging the factors and their constants,
the first factor satisfies the parameter-dependent division hypothesis of
Proposition~\ref{prop:product-Gleason} with the second factor as parameter
domain and with constant $K_1$. Then for every
commuting tuple $T\in\cR_{\Omega}(n)$ and every $f\in H^\infty(\Omega)$,
\[
  \|f(T)\|
  \leq
  C_{\Omega,n}\,\|f\|_{H^\infty(\Omega)},
\]
where $C_{\Omega,1}=1$ and, for $n\geq 2$,
\[
  C_{\Omega,n}=(2dK_{\Omega})^{n-1},
  \qquad
  K_{\Omega}=\max\{2K_1,K_2\}.
\]
\end{corollary}

\begin{proof}
By Lemma~\ref{lem:product-nc-ball}, the levelwise product of the two NC
operator balls is an NC operator ball, with
$\rho_{\cR_{\Omega}}=p_1q_1+p_2q_2=d$ by
Proposition~\ref{prop:cartan-nc-automorphisms}. By that same proposition each
factor has a transitive liftable subgroup, and their componentwise action gives
\textup{(A1)} for the product. The bounded
$H^\infty$-Gleason property for $\Omega$ follows from
Proposition~\ref{prop:product-Gleason}, with product constant
$K_{\Omega}=\max\{2K_1,K_2\}$. Thus $\cR_{\Omega}$ satisfies both
\textup{(A1)} and \textup{(A2)}, and the conclusion follows from
Theorem~\ref{thm:scaled-von-neumann-recipe}.
\end{proof}

\begin{remark}
Corollary~\ref{cor:row-column-product} is unconditional because its scalar
factors are Euclidean balls. Corollary~\ref{cor:cartan-conditional}, by
contrast, assumes the bounded $H^\infty$-Gleason property at the origin. The
NC operator-ball structure and automorphism condition \textup{(A1)} are
already established for these domains.
\end{remark}

\section*{Acknowledgements}

\paragraph{AI disclosure.}
AI tools were used to assist with wording, LaTeX preparation, notation,
literature searches, and reference checks, including searches concerning the
bounded $H^\infty$-Gleason property for Cartan domains. The mathematical ideas
and proofs were developed by the author, who takes full responsibility for
the content and correctness of the paper.

\bibliographystyle{plainnat}
\bibliography{refs}

\end{document}